\documentclass[reqno]{amsart}
\usepackage[normalem]{ulem}
\usepackage{hyperref} 
\usepackage{color}
\usepackage{latexsym}
\usepackage{graphicx}
\usepackage{mathrsfs}
\usepackage{amssymb}
\usepackage{amsfonts}
\usepackage{amssymb,amsmath,amsthm,tikz}
\newcommand{\vol}{{\text vol}}

\newcommand{\De}{\Delta}

\hypersetup{colorlinks = true, urlcolor = black}
\newcommand{\red}{\color{red}}

\newcommand{\edit}[1]{ {\red \cs #1 \cs}}
\newcommand{\cs}{$\clubsuit$}

\DeclareMathOperator{\Vol}{Vol}

\DeclareMathOperator{\Hilb}{Hilb}

\newcommand{\Sjostrand}{Sj\"ostrand }

\newcommand{\C}{\mathbb{C}}
\newcommand{\E}{\mathbb{E}}

\newcommand{\CP}{\mathbb{CP}}

\newcommand{\R}{\mathbb{R}}

\newcommand{\Z}{\mathbb{Z}}

\newcommand{\N}{\mathbb{N}}

\renewcommand{\P}{\mathbb{P}}

\newcommand{\acal}{\mathcal{A}}

\newcommand{\ccal}{\mathcal{C}}
\newcommand{\dcal}{\mathcal{D}}

\newcommand{\fcal}{\mathcal{F}}

\newcommand{\hcal}{\mathcal{H}}

\newcommand{\lcal}{\mathcal{L}}
\newcommand{\mcal}{\mathcal{M}}
\newcommand{\ncal}{\mathcal{N}}
\newcommand{\ocal}{\mathcal{O}}
\newcommand{\pcal}{\mathcal{P}}

\newcommand{\rcal}{\mathcal{R}}
\newcommand{\scal}{\mathcal{S}}

\newcommand{\wh}{\widehat}
\newcommand{\wt}{\widetilde}
\newcommand{\wb}{\overline}

\newcommand{\bma}{\begin{bmatrix}}
\newcommand{\ema}{\end{bmatrix}}
\newcommand{\baa}{\begin{align*}}
\newcommand{\eaa}{\end{align*}}
\newcommand{\bea}{\begin{eqnarray*} }
\newcommand{\eea}{\end{eqnarray*} }
\newcommand{\bee}{\begin{eqnarray} }
\newcommand{\eee}{\end{eqnarray} }
\newcommand{\be}{\begin{equation} }
\newcommand{\ee}{\end{equation} }
\newcommand{\bp}{\begin{prop}}
\newcommand{\ep}{\end{prop}}
\newcommand{\bt}{\begin{theorem}}
\newcommand{\et}{\end{theorem}}
\newcommand{\bpf}{\begin{proof}}
\newcommand{\epf}{\end{proof}}
\newcommand{\bl}{\begin{lem}}
\newcommand{\el}{\end{lem}}
\newcommand{\bc}{\begin{cor}}
\newcommand{\ec}{\end{cor}}
\newcommand{\bd}{\begin{defn}}
\newcommand{\ed}{\end{defn}}
\newcommand{\bcs}{\begin{cases}}
\newcommand{\ecs}{\end{cases}}
\newcommand{\bex}{\begin{example}}
\newcommand{\eex}{\end{example}}
\newcommand{\brem}{\begin{rem}}
\newcommand{\erem}{\end{rem}}

\newcommand{\pa}{\partial}
\newcommand{\ot}{\otimes}
\newcommand{\half}{\frac{1}{2}}
\renewcommand{\d}{\partial}
\newcommand{\dbar}{\bar\partial}
\newcommand{\ddbar}{\partial\dbar}
\newcommand{\RM}{\backslash}

\newcommand{\la}{\lambda}

\newcommand{\hPi}{\h \Pi}

\newcommand{\h}{\hat} 

\newcommand{\lan}{\langle}
\newcommand{\ran}{\rangle}

\def\XXint#1#2#3{{\setbox0=\hbox{$#1{#2#3}{\int}$ }
\vcenter{\hbox{$#2#3$ }}\kern-.6\wd0}}

\newtheorem{theo}{{\sc Theorem}}[section]
\newtheorem{cor}[theo]{{\sc Corollary}}

\newtheorem{lem}[theo]{{\sc Lemma}}
\newtheorem{prop}[theo]{{\sc Proposition}}
\newtheorem{defn}[theo]{{\sc Definition}}
\newtheorem{rem}{{\sc Remark}}
\newenvironment{example}{\medskip\noindent{\it Example:\/} }{\medskip}

\newcommand{\QQ}{\mathcal{Q}}

\newcommand{\T}{{\mathbf T}^m}

\newcommand{\Hess}{{\operatorname {Hess}}}

\newcommand{\szego}{Szeg\"o }

\newcommand{\kahler}{K\"ahler }
\newcommand{\Kahler}{K\"ahler }

\title{Central Limit theorem for toric \kahler manifolds}

\author{Steve Zelditch and Peng Zhou}
\address{Department of Mathematics, Northwestern  University, Evanston, IL 60208, USA}

\email{zelditch@math.northwestern.edu}

\thanks{Research partially supported by NSF grant  DMS-1541126
and by the Stefan Bergman trust  .}

\date{\today}

\begin{document}

\begin{abstract} Associated to the Bergman kernels of a polarized toric
\kahler manifold $(M, \omega, L, h)$  are  sequences of measures $\{\mu_k^z\}_{k=1}^{\infty}$
parametrized by the points $z \in M$. For each $z$ in the open orbit, we  prove a central limit theorem for $\mu_k^z$. The center of mass of  $\mu_k^z$ is the image
of $z$ under the moment map;  after re-centering at $0$ and
dilating by $\sqrt{k}$, the re-normalized measure tends to a centered Gaussian whose variance is the Hessian of the \kahler potential at $z$. We further give
a remainder estimate of Berry-Esseen type. The sequence $\{\mu_k^z\}$  is generally not a sequence of convolution powers and the proofs only involve \kahler analysis. 

\end{abstract}

\maketitle


\section{Introduction}

Let $(L, h,  M,\omega)$ be a polarized toric \kahler manifold with ample toric line bundle $L \to M$. Thus, there exists a Hamiltonian torus action $\Phi^{\vec t}(z): \T \times M \to M$ on $M$ which extends holomorphically to a $(\C^*)^m$ action, and 
$M $ is the closure of an open orbit $M^o = (\C^*)^m \{z_0\}$.  Let $h$ denote 
a $\T$-invariant
 Hermitian metric on  $L $ with curvature form $\omega$. The moment map 
\begin{equation} \label{MM} \mu_h:= \mu:  M \to P \subset \R^m, \end{equation}   associated to this data defines a torus bundle on the open orbit over a convex  lattice polytope  $P$ known as a  Delzant polytope. As reviewed in Section \ref{MONSECT}, there is a natural basis $\{s_{\alpha}\}_{\alpha \in k P}$  of the space $H^0(M, L^k)$ of holomorphic sections of the $k$-th power of $L$ by 
eigensections $s_{\alpha}$ of the $\T$ action. In a standard frame $e_L$ of $L$ over $M^o$, they correspond to monomials $z^{\alpha}$ on $(\C^*)^m$. For any $z \in M^o$ and $k \in \N$, we define the probability measure,
\begin{equation} \label{MUKZDEF} \mu_k^z = \frac{1}{\Pi_{h^{k }}(z,z)}\;\; \sum_{\alpha \in kP \cap \Z^m}
\frac{|s_{\alpha}(z)|_{h^{k }}^2}{\|s_{\alpha}\|_{h^{k }}^2}   \;
\delta_{\frac{\alpha}{k }} \in \mcal_1(\R^m),
\end{equation} 
on $\R^m$.   Here, $\|s_{\alpha}\|_{h^k }$ is the $L^2$ norm of
 $s_{\alpha}$ with respect to the natural inner product $\text{Hilb}_k(h)$ induced by the Hermitian metric on $H^0(M, L^k)$ and   $\Pi_{h^{k }}(z,z)$ is the contracted \szego kernel on the
diagonal (or density of states); see \S \ref{XhSECT}  for
background. The measures are  discrete measures supported on
  $P \cap \frac{1}{k} \Z^m$, and 
 were previously studied in \cite{SoZ10,SoZ12}. The main result of this article is that for each $z, $  the sequence $\{\mu_k^z\}_{k =1}^{\infty}$ satisfies a CLT (central limit theorem) with a Berry-Esseen type remainder estimate.

To state the results precisely we need to introduce some notation and background. A Gaussian measure on $\R^m$ with mean $\vec m$
and covariance matrix $\Sigma$ is a measure of the form,
$$\gamma_{\vec m, \Sigma} (\vec x): = (2 \pi \det \Sigma)^{-m/2} e^{- \half
\langle \Sigma^{-1} (\vec x - \vec m), (\vec x - \vec m) \rangle}. $$
Our aim is to prove that in the sense of weak convergence, dilations of \eqref{MUKZDEF} tend to a certain Gaussian measure,
\begin{equation}\frac{1}{\Pi_{h^{k }}(z,z)}\;\; \sum_{\alpha \in kP \cap \Z^m}
\frac{|s_{\alpha}(z)|_{h^{k }}^2}{\|s_{\alpha}\|_{h^{k }}^2}   \;
\delta_{\sqrt{k} (\frac{\alpha}{k} - \mu_h(z))} \to \gamma_{0, \text{Hess} \varphi (z)}, \end{equation}
whose mean is $0$ and whose covariance matrix is the Hessian
${\Hess} \varphi(z)$
of the \kahler potential. 
In Section \ref{BACKGROUND}, we review the fact that the \kahler potential $\varphi(z) $ of a toric variety
  is  a convex function of $(\rho_1, \cdots, \rho_m) =   (\log |z_1|^2, \cdots, \log |z_m|^2)$ on   $\R^m
$.
Here, we use orbit coordinates $(\rho, \theta)$ where $e^{\rho_i} = |z_i|^2$. Then  $\nabla_{\rho}\varphi$ is the gradient, resp.   ${\Hess} \; \varphi = \pa_{\rho_i \rho_j}^2 \varphi(e^{\rho/2})$ is the Hessian in the $\rho$ variables. We refer to Section \ref{BACKGROUNDOO}
for definitions and details.

\subsection{Mean and covariance}
To determine the appropriate Gaussian measure we need to determine  the
asymptotics as $k \to \infty$ of the mean, resp.  covariance matrix \begin{equation}
\vec m_k(z) = \int_P \vec x d\mu_k^z(x), \; \text{resp.} \; \;  [\Sigma_k]_{ij}(z) = \int_P (x_i - m_{k, i}(z))(x_j - m_{k, j}(z)) d\mu_k^z. \end{equation}

\begin{lem} \label{MVAR} Let $\mu_h: M \to P$ be the moment map \eqref{MM}. Then,
$$\vec m_k(z)  = \mu_h(z) + O(1/k), \;\;\; \Sigma_k(z) = \frac{1}{k} {\Hess} \;\varphi(z) + O(\frac{1}{k^2}) $$
\end{lem}

\noindent  
 The proof  is reviewed in  Section \ref{MVSECT} (from \cite[Proposition 6.3]{Z09}). It implies the  law of large numbers for the sequence $\{\mu_k^z\}$: In the weak topology of measures on $C(\bar{P})$,
$\mu_k^z \to \delta_{\mu_h(z)}. $ 
We therefore center the measures   \eqref{MUKZDEF} at
$\mu(z)$, i.e. put
\begin{equation} \label{CENTERED} \tilde{\mu}_k^z = \mu_k^z (x - \mu_h(z)), \end{equation}
and then dilate by $\sqrt{k}$ to obtain the normalized sequence,

\begin{equation}
\label{DILATEmukz} D_{\sqrt{k}} \tilde{\mu}_k^z 
= \frac{1}{\Pi_{h^{k }}(z,z)}\;\; \sum_{\alpha \in k P \cap \Z^m}
\frac{|s_{\alpha}(z)|_{h^{k }}^2}{\|s_{\alpha}\|_{h^{k }}^2}   \;
\delta_{\sqrt{k} (\frac{\alpha}{k} - \mu_h(z))}. \end{equation} Equivalently, if $f \in C_b(\R^m)$. Then,
\begin{equation}\label{MEASURE} \langle f, D_{\sqrt{k}}\tilde{ \mu}_k^z \rangle
= \frac{1}{\Pi_{h^{k }}(z,z)}\;\; \sum_{\alpha \in kP \cap \Z^m}
\frac{|s_{\alpha}(z)|_{h^{k }}^2}{\|s_{\alpha}\|_{h^{k }}^2}   \;
f(\sqrt{k} (\frac{\alpha}{k}- \mu_h(z)),
\end{equation} 
Here, $C_b(\R^m)$ denotes the space
of bounded continuous functions on $\R^m$.

\subsection{Weak* convergence on $C_b(\R^m)$}
  Our first main result is the following

\begin{theo}  \label{CLT}  In the topology of weak* convergence on $C_b(\R^m)$, $$D_{\sqrt{k}} \wt \mu_k^z \stackrel{w*}{\rightarrow} \gamma_{0, {\Hess}\; \varphi (z)}.
 $$
That is, for any $f \in C_b(\R^m)$, 
$$\int_{\R^m} f(x) D_{\sqrt{k}}d \wt \mu_k^z  (x)\to \int_{\R^m} f(x) d \gamma_{0, {\Hess}\; \varphi (z)}(x). $$
 \end{theo}

The role of the parameter $z$ is similar to that of the parameter 
$p$ in the Bernoulli measures $\mu_p = p \delta_0 + (1-p) \delta_1$ and their convolution powers on the unit interval $[0,1]$. In very special cases, such as the Fubini-Study metric $h$ of
$M = \CP^m$, $\mu_k^z$ is itself a sequence of
dilated convolution powers, 
$\mu_k^z = (\mu_1^z)^{*k}  =  \mu_1^z  * \mu_1^z \cdots * \mu_1^z $ ($k$ times).  It has been pointed out in \cite{D08, STZ03} that  such a situation occurs for the Fubini-Study metric on
the dual hyperplane line bundle $\ocal(1) \to \CP^m$ and for the Bargmann-Fock case.   This is equivalent to 
the condition that $\Pi_{h^k} = (\Pi_{h^1})^k$. In \cite{STZ03} the relation between $\Pi_{h^k}$ and $(\Pi_{h^1})^k$ on a toric \kahler manifold  has been given in terms of partition functions of lattice random walks and a certain pseudo-differential operator. To the author's knowledge, conditions on $h$ which are necessary and/or sufficient that $\Pi_{h^k}= (\Pi_{h^1})^k$ are not known even for a toric variety. Donaldson points out that it holds for sequences of metrics defined
by Veronese embeddings. \footnote{What Donaldson calls the CLT in \cite[(9)]{D08} is a local limit law of the kind proved in \cite{SoZ07} (see Section \ref{LLSECT}). The Poisson limit law alluded to in \cite{D08} was proved in \cite{SoZ10,F12}.} The proof of Theorem \ref{CLT} does not appeal to any prior results on central limit theorems or probablility theory but is purely a result of toric \kahler analysis and is based on the quantum dynamics of the torus action (see \eqref{ucaldef}).

\subsection{Berry-Esseen type remainder}

The classical Berry-Esseen theorem  gives a quantitative remainder estimate
for the CLT for  sums $S_N = X_1 + \cdots + X_N$ of i.i.d. real-valued  random variables with finite
third moment.  With no loss of generality, assume that $\E X_j = 0, \text{Var}(X_j) = 1$ and let $m_3 = \E |X_j|^3. $  Let $\mu$ denote the common distribution of the $X_j$. Then the Berry-Esseen remainder bound states that 
if $f$ is a ``$\gamma_{0,1}$-continuous bounded function'' then
\begin{equation} \label{BE} \int_{\R^m} f(x) D_{\sqrt{k}} d\mu^{*k}(x)  = \int_{\R^m} f(x) d\gamma_{0,1}(x) + O(\frac{m_3}{\sqrt{k}}). \end{equation}
Such functions include characteristic functions of sets whose boundaries
have Lebesgue measure zero.
The Berry-Esseen bound was extended to the multivariate CLT by Bergstrom,  Bhattacharya,  Rotar, Sazonov and von Bahr  around 1970; see \cite{Bhat} for background and references.  
The measures $\mu_k^z$ of this article would be referred to as distributions
of lattice random variables
$\vec X_k$, i.e. random variables whose values are almost surely located
on lattice points of $\frac{1}{k} \Z^m \cap P$.  Special techniques are available for lattice random variables (see \cite[Chapter 5]{Bhat}) but we do not use them here.
The following is a simple analogue of the  remainder estimate of \cite{ZZ17b},
and is stated for certain continuous test functions  rather than for characteristic functions of sets.

\begin{theo}  \label{BE2} If $f \in C_0(\R^m)$,  with $\hat{f} \in L^1(\R^m)$ bounded by a radially decreasing $L^1$ function,  then  $$\int_{\R^m} f(x) D_{\sqrt{k}}d \wt \mu_k^z(x)= \int_{\R^m} f  (x) d\gamma_{0,{\Hess}\; \varphi (z)}(x) + O_f(\frac{1}{\sqrt{k}}). $$

 \end{theo}
 
 The analogous remainder estimate is proved for $S^1$ actions with scalar moment map in  \cite[Theorem 2]{ZZ16} and we generalize the periodization argument from that article.  

\subsection{Local Limit theorem}

In this section, we tie together some results of   \cite{SoZ10, SoZ12}    to Theorem \ref{CLT}. We show that the former  imply a ``local limit law'' for the dilated measures $\mu_k^{z,1}(p) : = \mu_k^z(p/k)$ for $p \in \R^m$.

Classically, a  local limit theorem for lattice random variables pertains
to a triangular array 
$\{X_{N,k}\}_{k=1}^N$ of independent random variables with 
values in $\Z^m$. Let $S_N = \sum_{k=1}^N X_{N,k}$.  Assuming that $\frac{S_N - \E S_N}{\sqrt{N}} \to \gamma_{0, \Sigma}$ in the weak-* sense,  the  local limit theorem
states that
$${\mathbb P}_N(\alpha) : = {\mathbb P}\{S_N = \alpha\} = N^{-m/2} \gamma_{0, \sigma^2}(\frac{\alpha - \E S_N}{\sqrt{N}}) + o(N^{-m/2}). $$
For instance, in the model case of  Bernoulli random variables with $\P(X_1 = 1) = \half =
\P(X_1 = 0)$,  
$\P(S_N = k) = {N \choose k} 2^{-N}$ and  $$P(S_N = k) \simeq \frac{\sqrt{2}}{\sqrt{\pi N}} e^{- \frac{(k - N/2)^2}{N/2}}. $$
We refer to \cite[Chapter 9]{GK} and \cite{Muk91} for discussion of local limit theorems
for lattice distributions.

These results do not apply to the measures \eqref{MUKZDEF}. However, we  prove that  they  satisfy the following  local limit theorem:
\begin{theo}\label{LLT}
 for $\alpha \in \Z^m$,
\begin{equation} \label{LLL} \mu_k^{z,1}(\alpha) = k^{-m/2} \gamma_{0, \Hess \varphi(z)} (\sqrt{k} \left(\frac{\alpha}{k} - \mu_h(z)\right)) (1+ O(1/k)). \end{equation}
\end{theo}

All of the necessary calculations and esimtates were proved in \cite{SoZ10,SoZ12}, but the conclusion was not drawn there.

  In the classical case of independent lattice variables, such as the de Moivre-Laplace theorem,  the CLT can be
derived from the local limit law by integrating (i.e. summing) the latter.
The localization formulae in \cite{SoZ10} could probably be used to prove Theorem \ref{CLT} from Theorem \ref{LLT} in  this way. But the proof we give of Theorem \ref{CLT} seems simpler as well as giving a sharper remainder estimate.

\subsection{\label{RR}Related results}

  Theorem \ref{CLT}  some resemblence in both its statement and proof to  the CLT proved in \cite{ZZ17b} for Hamiltonian flows and in \cite{ZZ16} for $S^1$ actions.  See also \cite{PS,RS} for prior articles with related results. But these  articles involve sequences of probability measures on $\R$, while the CLT in this article is about the sequence $\mu_k^z$ of probability measures on $P \subset \R^m$. Moreover, those articles gave Erf asymptotics
for scaled partial Bergman kernels around the interface $\partial \acal$  in $M$  between an allowed region $\acal$ and its complement. This article gives a  vector-valued refinement of the CLT of \cite{ZZ16} in which 
$\mu_h^{-1}(z)$ is a single torus instead of a hypersurface $\partial \acal$, and  Gaussian asymptotics hold in all normal directions to the torus.

In Theorem \ref{CLT}, we  assume that $z \in P^o$, the interior of $P$, and show that  
the limit  is uniform on compact
subsets of $M^o$.  If we allow varying points $z_k \to \partial P$,  then as in the model binomial case,
the measures $\mu_k^{z_k}$ tend to some kind of  Poisson limit law. Results of this kind are  proved in \cite{SoZ10,F12} in the toric setting. It would be interesting to investigate such Poisson limit laws on  general \kahler manifolds and for general Hamiltonian, where  $\dcal$ is replaced by the set of critical points of $H$. Critical levels were excluded in \cite{ZZ16,ZZ17b}.

 An intriguing question is whether Theorem \ref{CLT} admits a generalization to non-toric \kahler manifolds. One possibility is to try  to adapt it to the other
 \kahler manifolds of large symmetry discussed in \cite{D08}. Another is to try to define analogues of $\mu_k^z$ on Okounkov bodies of polarized \kahler manifolds. In the latter case, even the law of large numbers does not seem to have been formulated.

\section{\label{BACKGROUND} Background on toric varieties}

We employ the same notation and terminology as in
\cite{Z09,SoZ10, SoZ12}. We recall that a toric \kahler manifold is a \kahler manifold
$(M, J, \omega)$ on which the complex torus $(\C^*)^m$ acts
holomorphically with an open orbit $M^o$.  We
 choose a basepoint $z_0$ on the orbit open and
identify  $M^o \equiv (\C^{*})^{m} \{z_0\}$. The underlying real torus is
denoted $\T$ so that $(\C^*)^m = \T \times \R_+^m$, which we write
in coordinates as $z = e^{\rho/2 + i \theta}$ in a multi-index
notation.

We assume that $M$ is a smooth projective toric \kahler manifold, hence that  $P$ is a Delzant
polytope, i.e. that  $P$ is
 defined by a set of linear inequalities
$$l_r(x): =\langle x, v_r\rangle-\alpha_r \geq 0, ~~~r=1, ..., d, $$
where $v_r$ is a primitive element of the lattice and
inward-pointing normal to the $r$-th $(n-1)$-dimensional face of
$P$. We denote by $P^o$ the interior of $P$ and by $\partial P$
its boundary; $P = P^o \cup
\partial P$.

\subsection{\label{MONSECT}Monomial basis of $H^0(M, L^k)$, norms and \szego kernels}

 A natural basis of the
space of holomorphic sections $H^0(M, L^k)$ associated to the
$k$th power of $L \to M$ is defined by the  monomials $z^{\alpha}$
where $\alpha$ is a lattice point in the $k$th dilate of the
polytope, $\alpha \in k P \cap \Z^m.$ That is, there exists an
invariant frame $e_L$ over the open orbit so that $s_{\alpha}(z) =
z^{\alpha} e_L$.
We equip $L$ with a toric Hermitian metric $h$ whose
curvature $(1,1)$-form $\omega = i \ddbar \log \|e\|_{h}^2$ is positive.  We often express the norm in terms of a local
\kahler potential, $\|e\|_{h}^2 = e^{- \varphi}$, so that
$|s_{\alpha}(z)|_{h^k}^2 = |z^{\alpha}|^2 e^{- k \varphi (z)}$ for
$s_{\alpha} \in H^0(M, L^k)$.

Any  hermitian metric $h$ on $L$ induces inner products
$\Hilb_k(h)$ on $H^0(M, L^k)$, defined by
\begin{equation} \label{HILB} \langle s_1, s_2 \rangle_{\text{Hilb}_k(h)} =
\int_M (s_1(z), s_2(z))_{h^k} \frac{\omega_h^m}{m!}.
\end{equation} The monomials are orthogonal with respect to any
such toric inner product and have the norm-squares
\begin{equation} \label{QFORM} Q_{h^k}(\alpha) = \int_{\C^m} |z^{\alpha}|^2 e^{-
k \varphi(z)} dV_{\varphi}(z), \end{equation} where $dV_{\varphi} = (i
\ddbar \varphi)^m/ m!$.  We denote the dimension of $H^0(M, L^k)$ by $N_k$.

\subsection{\label{BACKGROUNDOO} \kahler potential, moment map  and sympletic potential} 
Recall that we use log coordinate $(\rho, \theta)$ on $M^o \cong (\C^*)^m$ by setting $z_i = e^{\rho_i/2 + \sqrt{-1} \theta_i}$. 
Since the Kahler potential $\varphi$ is $\T$-invariant, $\varphi(z)$ only depends on the $\rho$ variables, hence we may write it as $\varphi(\rho)$.

The moment map $\mu_h$ is defined as the gradient of the \kahler potential $\varphi:\R^m \to \R$. Let $\R^m_p$ be the dual space of $\R^m_\rho$, where we use coordinates $p=(p_1, \cdots, p_m)$ and $\rho=(\rho_1, \cdots, \rho_m)$ respectively. The gradient map induced by $\varphi$ is defined by 
\[ \Phi_\varphi: \R^m_\rho \to \R^m_p, \quad \rho \mapsto p(\rho) := (\pa_{\rho_1} \varphi, \cdots, \pa_{\rho_m} \varphi). \]
The moment map is then defined by,
\begin{equation}  \mu_h(z) = \Phi_\varphi(\rho). \label{symp-pot} \end{equation}

The  moment map $\mu_h: M \to \R^m$ is only well-defined up to an additive constant.  The  equivariant toric line bundle $L$ fixes this degree of freedom 
as follows:
Let $I_k \subset \Z^m$ be the subset consisting of weight $H^0(M, L^k)$ under the action of $(\C^*)^m$, and let $P_k$ be the convex hull of $I_k$. Then  $P_k = k P'$ for a fixed convex polytope $P'$. We normalize $\mu_h$ by requiring that $\mu_h(M) = P'$.  
For background, see \cite{Fu}.
%
%
%
%

\subsection{\label{XhSECT} The \szego kernel and the Bergman kernel}

 The \szego (or Bergman) kernels of a positive Hermitian line
bundle $(L, h) \to (M, \omega)$  are the kernels of the orthogonal
projections $\Pi_{h^k}: L^2(M, L^k) \to H^0(M, L^k)$ onto the
spaces of holomorphic sections with respect to the inner product
$\Hilb_k(h)$,
\begin{equation}\label{Pik} \Pi_{h^k} s(z) = \int_M \Pi_{h^k}(z,w) \cdot s(w)
\frac{\omega_h^m}{m!}, \end{equation} where the $\cdot$ denotes
the $h$-hermitian inner product at $w$.
 In terms of a local frame  $e$  for $L \to M$ over an
open set $U \subset M$,  we may write sections as $s = f e$. If
$\{s^k_j=f_j e_L^{\otimes k}:j=1,\dots,N_k\}$ is  an orthonormal
basis for $H^0(M,L^k)$, then  the \szego kernel can be written in
the form
\begin{equation}\label{szego}  \Pi_{h^k}(z, w): = F_{h^k}
(z, w)\,e_L^{\otimes k}(z) \otimes\overline {e_L^{\otimes
k}(w)}\,,\end{equation} where
\begin{equation}\label{FN}F_{h^k}(z, w)=
\sum_{j=1}^{N_k}f_j(z) \overline{f_j(w)}\;, ~~~N_k = \dim H^0(M, L^k).\end{equation}
We also introduce the local kernel $B_k(z,w)$, defined with respect to the unitary frame: 
\begin{equation} \label{SzK} 
\Pi_{h^k}(z,w) = B_k(z,w) \cdot \frac{e_L^k(z)}{\|e^k_L(z)\|_h} \ot \overline{\frac{e_L^k(w)}{\|e^k_L(w)\|_h}} 
\end{equation}
The  {\it density of states} $\Pi_{h^k}(z)$  is the contraction of $\Pi_{k}(z,w)$ with the hermitian metric on the diagonal, 
\[\Pi_{h^k}(z): =   \sum_{i=0}^{N_k}
\|s^k_i(z)\|_{h_k}^2= F_{h^k}
(z, z)\,\left |e(z) \right |^{2k}_{h} =\ B_k(z,z), \]
where in the first equality we record a standard abuse of notation in which the diagonal of the
\szego kernel is identified with its contraction. 
On the diagonal, we have the following asymptotic expansion the density of states,
\begin{equation} \label{TYZ}  \Pi_{h^k}(z) =  k^m + c_1 S(z) k^{m-1} + a_2(z) k^{m-2} +
\dots \end{equation} where  $S(z)$ is the
scalar curvature of $\omega$.

 \subsection{Bergman kernels for a toric variety} In the
case of a toric variety, we have
\begin{equation}\label{FNa}F_{h^k}(z, w)=
\sum_{\alpha \in k P \cap \Z^m}  \frac{z^{\alpha}
\bar{w}^{\alpha}}{Q_{h^k}(\alpha)} \;,\end{equation}   where $Q_{h^k}(\alpha)$ is defined in \eqref{QFORM}. If we sift out
the $\alpha$th term of $\Pi_{h^k}$  by means of Fourier analysis
on $\T$, we obtain
\begin{equation} \label{PHK} \pcal_{h^k}(\alpha, z): =
\frac{|z^{\alpha}|^2 e^{- k \varphi(z)}}{Q_{h^k}(\alpha)}.
\end{equation}

Let $\wt \varphi(z,w)$ denote the almost extension of $\varphi(z)$ from the diagonal, that is $\wt \varphi$ satisfies the  condition $\dbar^k_z\wt   \varphi(z,w)|_{z=w} = \d^k_w \wt \varphi(z,w)|_{z=w} = 0$ for all $k \in \N$ and $\wt \varphi(z,w)|_{z=w} = \varphi(z)$. The $\T$ action is by holomorphic
isometries of $(M, \omega)$ and therefore \begin{equation} \label{EQUI} \wt \varphi(\Phi^{\vec t} z,\Phi^{\vec t} w) = \wt \varphi(z,w).  \end{equation}

The \szego kernel \eqref{SzK} admits a parametrix with complex phase
$\wt \varphi$ (see e.g. \cite{BBSj}). In the case of a toric \kahler manifold,
it takes the following simple form \cite{STZ03}.

\begin{prop}\label{SZKTV}  For any hermitian  toric positive line bundle over a toric
variety, the \szego kernel for the metrics $h_{\varphi}^N$ have the
asymptotic expansions in a local frame on $M$,
$$B_{h^k}(z, w) \sim e^{k \left(\wt \varphi(z, w) - \frac{1}{2} (\varphi(z)
+ \varphi(w)) \right) } A_k(z,w) \;\; \mbox{mod} \; k^{- \infty},
$$ where $A_k(z,w) \sim k^m \left(1 + \frac{a_1(z,w)}{k} + \cdots\right) $ is a semi-classical symbol of order $m$ and where the phase satisfies
\eqref{EQUI}. \end{prop}

\subsection{\label{MVSECT} Proof of Lemma \ref{MVAR}}

As mentioned above, Lemma \ref{MVAR} was proved in \cite{Z09,SoZ10}. We briefly review the proof as preparation for the proof of Theorem \ref{CLT}.

\begin{prop} \label{COMPAREPIT2} Let $(M, L, h, \omega)$ be a polarized toric Hermitian line bundle. Then the means, resp. variances of $\mu_k^z$ are given respectively by,

\begin{enumerate}

\item $m_k(z)=   \mu_h(z) + O(k^{-1});$

\item $\Sigma_k(z) =  k^{-1}  {\Hess}\; \varphi +
 O(k^{-2}) $.

\end{enumerate}
\end{prop}

\begin{proof} We briefly review the proof. 
Recall that the Bergman density function $\Pi_{h^k}(z)$ is $\T$-invariant, hence is a function of $\rho$, and can be written as 
\[ \Pi_{h^k}(\rho) = \sum_{\alpha \in k P \cap \Z^m} \Pi_{h^k, \alpha}(\rho) = \sum_{\alpha \in k P \cap \Z^m}  \frac{e^{\lan \alpha, \rho \ran - k \varphi(\rho)}}{Q_k(\alpha)},\]
Thus by explicit calculation we have
\bea
k^{-1} \pa_{\rho_j} \Pi_{h^k}(\rho) &=& \sum_{\alpha \in k P \cap \Z^m} \left(\frac{\alpha_j }{k} - \pa_{\rho_j} \varphi(\rho) \right)  \frac{e^{\lan \alpha, \rho \ran - k \varphi(\rho)}}{Q_k(\alpha)} \\
&=& \Pi_{h^k}(\rho) ( m_k(z) - \mu_h(z))_j
\eea
where we used $\pa_{\rho_j} \varphi(\rho) = \mu_h(z)_j$. 
Using the asymptotic expansion for $\Pi_{h^k}(\rho)$, we get $m_k(z) = \mu_h(z) + O(1/k)$. 

Then for the variance, we use
\bea
k^{-2} \pa^2_{\rho_i \rho_j} \Pi_{h^k}(\rho) &=& \sum_{\alpha \in k P \cap \Z^m} \left(\frac{\alpha_i}{k} - \pa_{\rho_i} \varphi(\rho) \right)\left(\frac{\alpha_j }{k} - \pa_{\rho_j} \varphi(\rho) \right) \frac{e^{\lan \alpha, \rho \ran - \mu_h(z)}}{Q_k(\alpha)} \\
&& - \sum_{\alpha \in k P \cap \Z^m} \left(  \frac{1}{k} \pa^2_{\rho_i \rho_j } \varphi(\rho) \right)  \frac{e^{\lan \alpha, \rho \ran - k \varphi(\rho)}}{Q_k(\alpha)}
\eea
Then divide by $\Pi_{k}(\rho)$ and use $m_k(z) = \mu_h(z) + O(1/k)$, we get the desired result for variance. 
%
%
%
%
%
%
%
%
%
%
%
%

\end{proof}

\section{\label{QDSECT} Quantum dynamics: Proof of Theorem \ref{CLT} }

The proof is somewhat similar to that of \cite[Theorem 4]{ZZ16} but in fact simpler because of the extra degrees of symmetry of a toric variety. 
A key simplifying feature   is that, like the $S^1$ action of \cite{ZZ16}, 
$\T$ acts holomorphically on $M$. As above, denote the action by  \begin{equation} (e^{i \vec t}, z) \in   \T \times M \mapsto e^{i \vec t} \cdot z =: \Phi^{\vec t} (z)  \end{equation}
and denote the infinitesimal generators of the action by
$\frac{\partial}{\partial \theta_j}$.
As discussed in \cite{STZ03,SoZ10,SoZ12}, the torus action can be quantized as a sequence of unitary operators $U_k(\vec t)$ on $H^0(M, L^k)$, or more precisely as a semi-classical Toeplitz Fourier integral opertator. We briefly review the key ideas and refer to the articles above for details and further background.

Let  $X_h = \partial D^*_h$ where $D_h^*$ is the unit co-disc bundle in $L^*$ with respect to $h$ and let
$\hcal^2(X_h)
\subset \lcal^2(X_h)$ denote the Hardy space of $L^2$ Cauchy-Riemann functions on $X_h$. It  is an $S^1$ bundle $\pi: X_h \to M$ and carries the $S^1$ action $r_{\theta}: S^1 \times X_h \to X_h$ by rotation of the fibers.
The $S^1$ action on $X_h$ commutes
with $\bar{\partial}_b$; hence $\hcal^2(X) = \bigoplus_{k
=0}^{\infty} \hcal^2_k(X)$ where $\hcal^2_k(X) =
\{ F \in \hcal^2(X): F(r_{\theta}x)
= e^{i
k \theta} F(x) \}$. Section $s_k$ of $L^k$ lift to equivariant
functions
$\hat{s}_l$ on $L^*$ by the rule
$$\hat{s}_k(\lambda) = \left( \lambda^{\otimes k}, s_k(z)
\right)\,,\quad
\la\in L^*_z\,,\ z\in M\,,$$
where $\lambda^{\otimes k} = \lambda \otimes
\cdots\otimes
\lambda$. We henceforth
restrict
$\hat{s}$ to $X$ and then the equivariance property takes the form
$\hat s_k(r_\theta x) = e^{ik\theta} \hat s_k(x)$. The map $s\mapsto
\hat{s}$ is a unitary equivalence between $H^0(M, L^{ k})$ and
$\hcal^2_k(X)$.

 There is a natural contact 1-form
$\alpha$ on $X_h$ defined by the  Hermitian connection 1-form, which
satisfies $d \alpha = \pi^* \omega$. 
The   $\T$ action lifts to $X_h$ as an action of the torus by 
contact transformations.  The generators $\frac{\partial}{\partial \theta_j}$ of the
$\T$ action on $M$ lift to contact vector fields
$\Xi_1,\dots\Xi_m$ on $X$. The horizontal lifts of the
Hamilton vector fields $\xi_j$ are then defined by
$$\pi_* \xi^h_{j} = \xi_j,\;\;\; \alpha(\xi^h_j) = 0,$$
and the contact vector fields $\Xi_j$ are given by:
\begin{equation} \label{Xij} \Xi_j = \xi^h_j + 2 \pi i \langle
\mu \circ\pi, \xi_j^* \rangle \frac{\partial}{\partial \theta}
=\xi^h_j + 2\pi i (\mu \circ\pi)_j\, \frac{\partial}{\partial
\theta},
\end{equation}
where $\mu$ is the moment map.
These vector fields act as differential operators
$ \hat{I}_j:\hcal^2_N(X) \to
\hcal^2_N(X)$ satisfying
\begin{equation}\label{Xi}(\hat{I}_j \hat S)
(\zeta)= \frac{1}{i} \frac{\d}{\d\varphi_j}  \hat
S(e^{i\varphi}\cdot\zeta)|_{\varphi=0}\;,\quad \hat S\in
\ccal^\infty(X)\;.
\end{equation}
Furthermore, the generator of the $S^1$ action acts on these
spaces and
\begin{equation}\label{dtheta} \hat{I}_{m +1} :\hcal^2_k(X) \to
\hcal^2_k(X)\;,\qquad \frac{1}{i}\frac{\d}{\d\theta}\hat s_k = k
\hat s_k \quad \mbox{for }\ \hat s_N\in
\hcal^2_k(X)\;.\end{equation}

The monomial sections $s_{\alpha}$  (equal to
$z^{\alpha}$ on the open orbit)  lift to $\T \times S^1$ equivariant functions
$\hat{s}_{\alpha}$ on $X_h$, i.e. 
as  joint
eigenfunctions of the $(m + 1)$ commuting operators $\hat{I}_j$.

The lifts   $\hat{\Pi}_k(x,y) $ of the  \szego kernels \eqref{Pik} are the (Schwarz) kernel of the orthogonal projection
$\hat{\Pi}_k : \lcal^2(X)\rightarrow
\hcal^2_k(X)$. They are  Fourier components,
\begin{equation} \label{hatPik}  \hPi_{h^k}(x,y) = \int_0^{2\pi} e^{-i k \theta} \hPi(r_{\theta} x, y) \frac{d\theta}{2\pi}, \end{equation}
   of the full  \szego projector  $\hPi(x,y)$.

The quantum torus action is defined by  $$U_{h^k}(\vec t): = e^{\sum_{j=1}^m t_j \hat{I}_j} = \prod_{j=1}^m e^{i t_j \hat{I}_j } $$
on $\hcal_k^2(X_h)$.   
Since the torus acts holomorphically, it is simply given by
\begin{equation} \label{UHATFORM} \h U_k(\vec t, x, y) =  \hat{\Pi}_{h^k} \Phi^{\vec t} \hat{ \Pi}_{h_k}(x,y)=    \hat{\Pi}_{h_k}( x, \Phi^{\vec t} y). \end{equation} We are most interested in the diagonal $ \h U_k(\vec t, x, x)$. It is $S^1$-invariant and depends only on $z = \pi(x)$,
so we denote it by
\begin{equation} \label{ucaldef} U_k(\vec t, z, z): = B_{h^k}( z, \Phi^{ \vec t}  z) =  \sum_{\alpha \in kP \cap \Z^m}
\frac{|s_{\alpha}(z)|_{h^{k }}^2}{\|s_{\alpha}\|_{h^{k }}^2}   \;
e^{-i  \langle \vec t,  \alpha\rangle}. \end{equation}
Here and henceforth we use the identification of the base and lifted \szego kernels and torus actions. Literally speaking, the translation of sections on the base requires parallel translation;  but on the open orbit we may think of the sections as scalar functions.

\subsection{Proof of Theorem \ref{CLT}}
We prove Theorem \ref{CLT} by the classical Fourier method, which is based on the `continuity theorem' that weak convergence $D_{\sqrt{k}} \mu_k^z \to 
\gamma_{0, {\Hess}\; \varphi (z)}$ is equivalent to  pointwise convergence of
the Fourier transforms (`characteristic functions in probability language) as long as the pointwise limit is continuous at $0$ (see e.g. \cite[Theorem 9.5.2]{Res}).

It is obvious that 
\begin{equation} \fcal^{-1}_{x \to t} D_{\sqrt{k}} \wt \mu_k^z (\vec t) = B_k^{-1}(z) U_k(\frac{\vec t}{\sqrt{k}}, z, z) e^{i \sqrt{k} \lan \mu(z), t\ran} ,  \end{equation}
or equivalently,
\begin{equation} \label{MAINID} \langle f, D_{\sqrt{k}}\tilde{ \mu}_k^z \rangle = B_k^{-1}(z) \int_{\R^m}   \hat{f}(\vec t)   U_k(\frac{\vec t}{\sqrt{k}}, z, z) e^{i \sqrt{k} \lan \mu(z), t\ran} dt \end{equation}
Thus, the  key point is to study the pointwise scaling asymptotics of \eqref{ucaldef}. 

Let $H_z  = \Hess \varphi(z)= \frac{\partial^2}{\partial \rho_i \rho_j } \varphi |_z$ denote the Hessian
of $\varphi$.

\begin{prop} \label{PTWISE}$ B_k^{-1}(z)U_k(\frac{\vec t}{\sqrt{k}}, z, z) e^{i \sqrt{k} \lan \mu(z), t\ran}  \to \fcal^{-1} {\gamma}_{0, H_z}$ pointwise.   \end{prop}

 \begin{proof} We need to show that, for each $z \in M^o$,
$$B^{-1}_k(z) U_k(\frac{\vec t}{\sqrt{k}}, z,z) e^{z,i \sqrt{k} \langle \vec t, \mu_h(z) \rangle } =   B_k^{-1}(z)B_{k}(\Phi^{ \frac{\vec t}{\sqrt{k}}} z)  e^{i \sqrt{k} \langle \vec t, \mu_h(z) \rangle } 
\to \fcal^{-1} \gamma_{ 0 , H_z} (\vec t). $$

Substituting the Boutet-de-Monvel-Sjoestrand parametrix of Proposition \ref{SZKTV}  
 gives,
\begin{equation} \label{BKFORM} U_k(\frac{\vec t}{\sqrt{k}}, z,z)e^{i \sqrt{k} \langle \vec t, \mu_h(z) \rangle }  \sim  
e^{k (\wt \varphi(z, \Phi^{ \frac{\vec t }{\sqrt{k}}} z) - \varphi(z))} e^{ i \sqrt{k} \langle \vec t, \mu_h(z) \rangle }   A_k \big( z,
\Phi^{ \frac{\vec t }{\sqrt{k}}}z \big),\end{equation} where $\sim $ means that the difference is a
function which decays rapidly in $N$ along with its derivatives.
Such a remainder may be neglected if we only consider expansions
modulo rapidly decaying functions of $N$. 
Using the parametrix, this comes down to the statement that
\begin{equation} \label{BKFORMb}  e^{k (\wt \varphi( z, \Phi^{ \frac{\vec t }{\sqrt{k}}} z) - \varphi(z))}  e^{- i \sqrt{k} \langle \vec t, \mu_h(z) \rangle }   A_k \big( z,
(\Phi^{ \frac{\vec t }{\sqrt{k}}}z)\big)   \to \fcal \gamma_z(\vec t)
\end{equation}

Use the $\T$-invariance of $\varphi$, we get \eqref{EQUI}, i.e.
\[ \wt \varphi(\Phi^{c} z, \Phi^c w) = \wt \varphi(z,w), \]
which implies 
\[ \wt \varphi( z, \Phi^{ \frac{\vec t }{\sqrt{k}}} z) = \wt \varphi(\Phi^{ \frac{-\vec t }{2\sqrt{k}}} z, \Phi^{ \frac{ \vec t }{2\sqrt{k}}}z). \] 
The $\T$ can be extended to a $(\C^*)^m$ action, 
For $\tau \in \C^m$, let $e^{\tau} \in (\C^*)^m$ acts on $z \in M^o$ by multiplication. Then we may define the following funciton 
\[ \Psi(\tau):= \wt \varphi( e^{\tau} z, \wb {e^{\tau}} z),  \]
so that $\wt \varphi(z, \Phi^{ \frac{\vec t }{\sqrt{k}}}  z)  = \Psi(-i t / 2 \sqrt{k})$. 
Since $\wt \varphi(z,w)$ is almost holomorphic in $z$ and anti-holomorphic in $w$ when $z=w$, $\Psi(\tau)$ is almost holomorphic in $\tau$ at $\tau=0$. If $\tau \in i \R^m$, this corresponds to action of $\T$. We restrict to  $\tau = c \in \R^m$,  and then $\wb {e^{c}} z = e^c z$.  Recall that $z=e^{\rho/2+i\theta}$ and $\varphi$ is a function of $\rho$ only, 
\[ \Psi(c) = \varphi(e^c z) = \varphi(\rho) + 2 c \pa_{\rho} \varphi(\rho) + \half \lan \Hess_\rho\varphi(\rho) 2c, 2c \ran+ O(c^3). \]
Finally, using $\Psi$ is analytic at $c=\vec 0$, we have 
\[ \Psi(-i t / 2 \sqrt{k}) = \varphi(\rho) - (it/\sqrt{k}) \pa_{\rho} \varphi(\rho) - \half k^{-1}  \lan H_z \vec t, \vec t \ran+ O(k^{-3/2})\]
Also note that $\mu_h(z) = \pa_\rho(\varphi(\rho)$ from \eqref{symp-pot}, we then have 
\[ k \wt \varphi(z, \Phi^{ \frac{\vec t }{\sqrt{k}}}  z) -k  \varphi(z)) + i \sqrt{k} \langle \vec t, \mu_h(z) \rangle = - \half\lan H_z \vec t, \vec t \ran+ O(k^{-1/2}) \]

The amplitude $\tilde{A}_N$   has an expansion of the form,
$$\tilde{A}_k\big( z,
 e^{i \frac{\xi}{k}} z , 0, N \big) = k^m  + k^{m-1} a_1 + O(k^{m-1}),$$
 for various smooth coefficients $a_j(z)$; the first one is a universal
 constant. We conclude that
\begin{equation} k^{-m} U_k(\frac{\vec t}{\sqrt{k}}, x, x)   e^{ i \sqrt{k} \langle \vec t, \mu_h(z) \rangle }  \to  e^{- \half\lan H_z \vec t, \vec t \ran }. \end{equation}

 \end{proof}

The proof of Proposition \ref{PTWISE} actually shows that there is a   pointwise expansion asymptotic expansion to all orders, with remainders of
polynomial growth in $\vec t$. For this it suffices to carry out the Taylor expansions of the phase and amplitudes to higher order. We can then integrate the result against suitable test functions to obtain the   following result,  analogous to  \cite[Proposition 8.1]{ZZ16}:

\begin{prop}\label{UCALASYM} Let $z \in M_0$.  For $f \in \scal(\R^m)$
with $\hat{f} \in C_0^{\infty}(\R^m)$, 
$$\int_{\R^m}   \hat{f}(\vec t) B_k^{-1}(z) U_k(\frac{\vec t}{\sqrt{k}}, z, z) e^{}dt=   \int_{\R^m}   \hat{f}(\vec t) e^{-\half \langle H_z \vec t, \vec t  \rangle} d t + O(k^{-\half}),$$
 where $H_{z} = \Hess \;\varphi({\rho})$ is the Hessian of the toric \kahler potential.
 In fact, there exists a complete asymptotic expansion of the integral in powers of $k^{-\half}$, and the asymptotics are uniform on compact subsets of $M^o$. 

\end{prop}

\begin{proof}

For $f \in \scal(\R^m)$ with $\hat{f} \in C_0^{\infty}(\R^m)$,  \begin{equation}\label{IDS}\begin{array}{lll} \langle f, D_{\sqrt{k}}\tilde{ \mu}_k^z \rangle
& = &  \frac{1}{\Pi_{h^{k }}(z,z)}\;\; \sum_{\alpha \in kP \cap \Z^m}
\frac{|s_{\alpha}(z)|_{h^{k }}^2}{\|s_{\alpha}\|_{h^{k }}^2}   \;
f(\sqrt{k} (\frac{\alpha}{k}- \mu_h(z)) \\ && \\ & = &
B_k^{-1}(z)   \int_{\R^m} \hat{f}(\vec t) U_k(\frac{\vec t}{\sqrt{k}}, z, z) e^{ i \sqrt{k} \langle \vec t, \mu_h(z) \rangle } dt \\ &&\\
& = & B_k^{-1}(z) \int_{\R^m}    \hat{f}(\vec t)  B_{h^k}(z,\Phi^{ \frac{\vec t}{\sqrt{k}}}    z)e^{ i \sqrt{k} \langle \vec t, \mu_h(z) \rangle } dt 
\end{array}\end{equation} 
Since the integrand is compactly supported, we may  apply the pointwise limit  of Proposition \ref{PTWISE} to obtain the principal term. 
By  Taylor expanding the
factor $e^{ k^{-\half}
R_3(k, z)} $ one obtains an oscillatory integral with the same
phase and a remainder of order $k^{-\half}$.

\end{proof}

Further details will be given in the proof of Proposition  \ref{INTERFACEf}.
  
\subsection{Completion of the proof of Theorem \ref{CLT}}

Since $\scal(\R^m)$ is dense is $C_0(\R^m)$ (continuous functions vanishing at infinity, equipped with
the sup norm),  Proposition \ref{UCALASYM} implies  weak* convergence  of $D_{\sqrt{k}} \wt \mu_k^z \to \gamma_{0, {\Hess}\; \varphi (z)}$ on $C_0(\R^m)$ (continuous functions vanishing at infinity).
For weak* convergence on $C_b(\R^m)$ one needs tightness of the sequence $\mu_k^z$. The so-called Levy continuity theorem on $\R^m$ says that if $\mu_k$ is a sequence of probability measures on $\R^m$ and
$\hat{\mu}_k (t) \to \varphi(t)$ pointwise and $\varphi(t)$ is continuous at $0$,
then $\varphi(t) = \hat{\mu}(t)$ for some probability measure $\mu$ and $\mu_k \stackrel{w*}{\rightarrow}  \mu$ on $C_b$.   The continuity of $\varphi$ at $t =0$ implies tightness of the sequence $\mu_k$.  We refer to \cite[Theorem 9.5.2]{Res}  for background. 
Hence, $D_{\sqrt{k}} \mu_k^z \stackrel{w*}{\rightarrow}  \gamma_{0, {\Hess}\; \varphi (z)}$ in the sense of  weak* convergence on $C_b(\R^n)$ as long as 
 $\fcal_{x \to t} D_{\sqrt{k}} \mu_k^z  \to \fcal \gamma_{0, {\Hess}\; \varphi (z)}$ pointwise. 
 Proposition \ref{PTWISE} thus implies that the sequence is tight,
and and further implies weak* convergence
on $C_b(\R)$. Hence Theorem \ref{CLT} is proved.

It follows from Theorem \ref{CLT} and the Portmanteau theorem that
  $$\lim_{k \to \infty} D_{\sqrt{k}} \wt \mu_k^z (K)  =   \gamma_{0, {\Hess} \varphi (z)}(K), $$
  for any  convex subset of $\R^m$, or more generally for any `continuity
  set' such that $\gamma_{0, {\Hess}\; \varphi (z)}(\partial K) = 0$.

\section{Berry-Esseen remainder  estimate}The purpose of
  this section is to improve the limit formula of Theorem \ref{CLT} and the 
  expansion in Proposition \ref{UCALASYM}  by giving the remainder
  estimate of Theorem \ref{BE}. We aim to give a representative result rather than the most general one possible, and therefore restrict to a reasonably general class of continuous functions rather than indicator functions. 

For any bounded continuous function $f\in C_b(\R)$, we define a slight simplification of  \eqref{IDS},
\begin{equation} \label{Ifk} I_{k,f}(z):= k^{-m}  \sum_{\alpha \in k P \cap \Z^m}
\frac{|s_{\alpha}(z)|_{h^{k }}^2}{\|s_{\alpha}\|_{h^{k }}^2}   \;
f(\sqrt{k} (\frac{\alpha}{k}- \mu_h(z)).   \end{equation}
Since $\Pi_{h^k}(z,z) = k^m(1+O(k^{-1/2}))$, we have $$ \langle f, D_{\sqrt{k}}\tilde{ \mu}_k^z \rangle =  I_{k,f}(z) (1 + O(k^{-\half})), $$
and it suffices to prove the desired bound for  \eqref{Ifk}.

We now prove the Berry-Esseen remainder bound for integrals of
$\mu_k^z$ against certain types of $f \in C_0(\R)$ (functions vanishing
as $|x| \to \infty$).

\bp  \label{INTERFACEf}  Let  $f \in C_0$ have the properties that
$\hat{f} \in L^1$ and that $|\hat{f}(\vec t)| \leq  C g(|\vec t|)$ where $g(|\vec t|) \in L^1$ and is monotonically decreasing as a function of $|\vec t|$. Then,
\[ I_{k,f}(z) = \int_{\R} f(x) d \gamma_{0, {\Hess} \; \varphi (z)}(x)+O_f(\;k^{-1/2}).\]

\ep

\bpf 
We start again from the last formula of \eqref{IDS}.
 We note that $\vec t \to
\hat{\Pi}_{h_k}(\Phi^{ \frac{\vec t}{\sqrt{k}}}  x, r_{\theta} x)) $ is periodic with respect to the lattice  $2 \pi \sqrt{k} \Z^m$ (similarly
for the parametrix and remainder terms),  so the integrals converge
when $\hat{f} \in \scal(\R)$. We periodize $g(\vec t) =\hat{f}(\vec t) e^{- i \sqrt{k} \langle \vec t, \mu_h(z) \rangle } $ with respect to the lattice $2 \pi \sqrt{k} \Z^m $  by means
of the $\sqrt{k}$-periodization operator
\[ \pcal_{\sqrt{k}} g(\vec t) := \sum_{\ell \in \Z^m} g(\vec t+2\pi \sqrt{k} \ell), \;\; g \in \scal(\R^m). \]  The sum  converges
as long as $|g(\vec t) | \in L^1(\R^m)$  is bounded by a decreasing positive
$L^1$ function.  Hence, as long as $\hat{f}$ has this property,
$$ \pcal_{\sqrt{k}}(\hat{f} e^{- i \sqrt{k} \langle \vec t, \mu_h(z) \rangle }  ) 
= \sum_{\ell \in \Z^m} \hat f(\vec t+2\pi \sqrt{k} \ell) e^{- i \sqrt{k} \langle \vec t, \mu_h(z) \rangle }  e^{- 2 \pi i k \langle \ell, \mu_h(z) \rangle} =: e^{- i  \sqrt{k} \langle \vec t, \mu_h(z) \rangle}  \hat{F}_k(\vec t), $$
 with $\hat F_k(\vec t) =  \sum_{\ell \in \Z^m} \hat f(\vec t+2\pi \sqrt{k} \ell) e^{-2\pi i (k \langle \ell, \mu_h(z) \rangle)}.$
Then,
\bea \label{Izk2}
I_{k,f}(z) & = & k^{-m} \int_{\sqrt{k} [-\pi, \pi]^m} \hat F_k(\vec t)   e^{- i  \sqrt{k} \langle \vec t, \mu_h(z) \rangle} 
B_{h^k}(\Phi^{ \frac{\vec t}{\sqrt{k}}}  z, z) d \vec t.
 \eea

We then localize the last  integral using a smooth cutoff $\chi(\frac{\vec t }{(\log k)^2} )$, where $\chi \in C_0^{\infty}(\R^m)$  is supported in $(-1,1)^m$ and equals to $1$ in $(-1/2, 1/2)^m$.  When $\pi \sqrt{k} \geq |\vec t| \geq  (\log k)^2$, the
off-diagonal Bergman kernel $\Pi_{h^k}(\Phi^{ \frac{\vec t}{\sqrt{k}}}  z, r_{\theta} z))$ is rapidly decaying   at the rate
$O(e^{- (\log k)^2})$. Here, we use the standard off-diagonal estimate,  $|\Pi_{h^k}(z,w)|
\leq C k^{m} e^{- \beta \sqrt{k} d(z,w)}$ for certain $\beta, C > 0$ (see \cite{ZZ17b} for background). Hence,
\bea \label{Izk3}
I_{k,f}(z) & = & k^{-m} \int_{\R^m} \chi(\frac{\vec t }{(\log k)^2} )\; \hat F_k(\vec t)  e^{- i  \sqrt{k} \langle \vec t, \mu_h(z) \rangle} 
B_{h_k}(\Phi^{ \frac{\vec t}{\sqrt{k}}}  z,z)) d \vec t + O_f(k^{-\infty}), \eea
where the constant in $O_f(k^{-\infty})$  depends on $\|\hat{F}_k\|_{L^1(-\sqrt{k}, \sqrt{k})^m} =  \|\hat{f}\|_{L^1}$.

We then introduce the Boutet-de-Monvel-\Sjostrand parametrix \eqref{SZKTV} to get, 
 \bea \label{Izkb}
I_{k,f}(z) 
&=&   \int_{-\infty}^\infty  \chi(\frac{t }{(\log k)^2} )\; \hat F_k(t)     e^{- i \sqrt{k} \langle \vec t, \mu_h(z) \rangle }  e^{k \wt \varphi(e^{i t/\sqrt{k}}  z,  z) - k \varphi(z)} A_k(e^{i t/ \sqrt{k}} z,z)  dt  
 \\
 &+&  \int_{-\infty}^\infty \chi(\frac{t }{(\log k)^2} )\; \hat F_k(t)     e^{- i \sqrt{k} \langle \vec t, \mu_h(z) \rangle } R_k(e^{i t/ \sqrt{k}} z,z)  dt + O_f(k^{-\infty}).
\eea
By the parametrix construction,   $R_k \in k^{-\infty} C^{\infty}(M \times M)$, hence  the second term is
$O(k^{-\infty})$ and may be absorbed into the remainder estimate.

As in the proof of Proposition \ref{PTWISE}, the  phase function of  $I_{k,f}$ has the Taylor expansion (or asymptotic expansion),
\begin{equation} \label{PSIDEF2} \begin{array}{lll} \Psi(it,  z) & = &  -i \sqrt{k}\langle \vec t, \mu_h(z)\rangle + k \wt \varphi(e^{i t/\sqrt{k}}  z,  z) - k \varphi(z) \\ && \\ 
&=&   -\half \langle H_z \vec t, \vec t \rangle +  g_1(it, z),  
\end{array} \end{equation}
where
\be  g_1 =  O(k^{-1/2}  |t|^3). \label{g3g4}\ee

We substitute the Taylor expansion into the phase of  the first term of $I_{k, f}(z)$,
and also Taylor expand $e^{ g_1}$ to order $1$.  Let $e_1(x) = 1 - e^x$. Since $|\vec t| \leq (\log k)^2$  on the support of the integrand, $|g_1| \leq  C (\frac{(\log k)^6}{\sqrt{k}})$ on $|\vec t| \leq (\log k)^2$.  Since $e^x = 1 + e_1(x)$ where $e_1(x) \leq 2 x$ on $[0,  C (\frac{\log k)^6}{\sqrt{k}})]$,  $e^{g_1} = 1 + \tilde{g}_1$ where $\tilde{g}_1(k,t) \leq 2 g_1 \leq C_0 k^{-\half} (1 + |\vec t|^3)$ on $ [0, (\log k)^2]$.   

We get
\bea  I_{k,f}( z) &
= &   \int_{\R^m} \chi(\frac{\vec t}{ (\log k)^2})\hat F_k(\vec t) e^{ -\half \langle H_z \vec t, \vec t \rangle}  (1 + \tilde{g}_1)) dt + O_f(k^{-1/2}) \\&& \\
& = & \int_{\R^m} \chi(\frac{\vec t}{(\log k)^2})\hat F_k(\vec t) e^{  -\half \langle H_z \vec t, \vec t \rangle}  dt + O_f(k^{-1/2})
\eea
where $ \chi(\frac{t}{ (\log k)^2}) |\wt g_1| \leq C_0 k^{-1/2} (1+ |\vec t|^3)$ after integration against the Gaussian factor is of size $O(k^{-1/2})$.

Finally, we unravel the periodization $\hat{F}_k$ to  evaluate the first term.
\bea
&&\int_{\R^m} \chi(\frac{t}{(\log k)^2})\hat F_k(t) e^{-\half \lan H_z \vec t, \vec t \ran}  dt \\
&=& \int_{\R^m} \chi(\frac{t}{(\log k)^2})\hat f(\vec t) e^{-\half \lan H_z \vec t, \vec t \ran}  dt   \\
& + &\sum_{\ell \in \Z^m \RM 0} \int_{\R} \chi(\frac{\vec t}{(\log k)^2})\hat f(\vec t + 2 \pi \sqrt{k} \ell) e^{2\pi i k \langle \ell, \mu_h(z) \rangle -\half \lan H_z \vec t, \vec t \ran}  dt \\
&=& \int_{\R^m} \chi(\frac{\vec t}{(\log k)^2})\hat f(\vec t)  e^{-\half \lan H_z \vec t, \vec t \ran}  dt  + O(k^{-\half}\|\hat{f}\|).
\eea
In the $\ell \not= 0$ sum, we use that
$$\begin{array}{lll} \sum_{\ell \in \Z^m \RM 0} |\hat{f}(\vec t + 2 \pi \sqrt{k} \ell)|
& \leq & \sum_{\ell \in \Z^m \RM 0} g(| \vec t + 2 \pi \sqrt{k} \ell| )  \\&& \\
&\leq & C \int_{|\ell | \geq 1} g(| \vec t + 2 \pi \sqrt{k} \ell| )   d \ell \\ &&\\
& = & \frac{C}{\sqrt{k}} \int_{|y | \geq \sqrt{k}} g(| \vec t + 2 \pi y| )  d y \\ &&\\
&= &  C k^{-\half} \| g \|_{L^1}  , \end{array}$$
so that 
$$\sum_{\ell \in \Z^m \RM 0} \int_{\R} \chi(\frac{\vec t}{(\log k)^2})\hat f(\vec t + 2 \pi \sqrt{k} \ell) e^{2\pi i k \langle \ell, \mu_h(z) \rangle -\half \lan H_z \vec t, \vec t \ran} dt $$
is bounded by
$$[C' k^{-\half}\| g \|_{L^1} ]
\int_{\R} \chi(\frac{\vec t}{(\log k)^2})e^{-\half \lan H_z \vec t, \vec t \ran}   dt =  O(k^{-\half}\| g \|_{L^1} ).$$

Finally, removing the cut-off $\chi(t/(\log k)^2)$ from the $\ell = 0$ term introduces an error of order  $\int_{(\log k)^2}^\infty e^{-a x^2} dx = O(k^{-\infty})$. We have
\bea  I_{k,f}(z) &=& \int_{\R^m} \hat f(t) e^{ -\half \langle H_z  \vec t , \vec  t \rangle}  dt  + O_f(k^{-1/2}) \\
&=& \frac {1 }{(2\pi)^{m/2} \sqrt{ \det (H_z)}} \int_{\R^m} f(x) e^{-\frac{1}{2} \langle H_z^{-1} x, x \rangle}  dx +O_f(k^{-1/2})
\eea
by the Plancherel theorem.  This completes the proof of Theorem \ref{BE}.
\epf

\begin{rem} The result can be  generalized to indicator functions ${\bf 1}_K$ of convex sets $K$. It  would suffice to smoothe  ${\bf 1}_K$ and to measure
the error in the smoothing. The terms contributing to the latter are sums of 
\eqref{PHK} over lattice points close to $\partial K$. The size of the remainder thus depends on the position of $\mu_h(z)$ relative to $K$. 
\end{rem}

\section{\label{LLSECT}Local limit law: Proof of Theorem \ref{LLT}}

To prove Theorem \ref{LLT} we need to review some further background
on toric \kahler manifolds. 

Let $h = e^{-\varphi}$ be a toric Hermitian metric on $L$. Recall that 
the {\it symplectic potential} $u_{\varphi}$ associated to $\varphi$
 is its Legendre transform: for $x \in P$ there
is a unique $\rho(x)$ such that $\mu_{\varphi}(e^{\rho(x)/2}) =d\varphi (\rho(x))= x$. If $z = e^{\rho/2 + i \theta}$ then we write
$\rho_z = \rho = \log |z|^2$. Then the Legendre transform is defined to
be the convex function
\begin{equation} \label{SYMPOTDEF} u_{\varphi}(x) = \langle x,  \rho(x) \rangle -
\varphi(\rho(x)).
\end{equation} 
Also define
  \begin{equation} \label{Ikzdef} I^z(x) = u_{\varphi} (x) -
\langle x, \rho_z \rangle + \varphi (\rho_z).\end{equation} 
Then $I^z(x)$ is a convex function on $P$ with a minimum of value $0$ at $x = \mu_h(z)$ and with Hessian that of $u_{\varphi}$.

The weights $ \pcal_{h^k}(\alpha, z)$ \eqref{PHK} of the dilate
$\mu_k^{z,1}$ admit pointwise asymptotic expansions. 

\begin{lem} \label{pcalLem} $\pcal_{h^k}(\alpha, z) = k^{m/2} (2\pi)^{-m/2}|\det  {\Hess}(u_{\varphi}(\mu_h(z))|^{\half} e^{- k I^z(\frac{\alpha}{k})} (1+O(1/k)), $
where $O(1/k)$ is uniform in $z, \alpha$. 
\end{lem}

\bpf
For the sake of completeness, we briefly review the elements of the proof. 
In \cite{SoZ07,SoZ10}, the norming constants \eqref{QFORM} were 
evaluated in terms of the symplectic potential:
\begin{equation} \label{SPNORM} \QQ_{h^k}(\alpha) = \int_P e^{ k
(u_{\varphi}(x) + \langle \frac{\alpha}{k} - x, \nabla u_{\varphi}(x) \rangle}
d\Vol(x). \end{equation} By applying a steepest descent method, it was shown in    \cite[Proposition 3.1]{SoZ07} 
that for interior $\alpha \in k P$, 
\begin{equation} \label{QQ} \QQ_{h^k}(\alpha) = k^{-m/2} \frac{(2\pi)^{m/2}}{|\det \text{Hess} \; u_{\varphi}|^{\half} }   e^{ k
u_{\varphi} (\alpha/k)} (1+O(1/k)). \end{equation} 
Hence, in the coordinates $z = e^{\rho/2 + i\theta}$, 
$$\pcal_{h^k}(\alpha, z) = \frac{e^{\langle \alpha, \rho_z \rangle} e^{- k \varphi(z)}}{Q_{h^k}(\alpha)} \sim  k^{m/2} (2\pi)^{-m/2} |\det \text{Hess} \; u_{\varphi}|^{\half}e^{ -k
u_{\varphi}
(\alpha)}  e^{\langle \alpha, \rho_z \rangle} e^{- k \varphi(z)}, $$
as stated in the Lemma.
\epf

We now assume that $\frac{\alpha}{k} = \mu_h(z) + O(1/k)$ and have
$$I^z(\frac{\alpha}{k}) \simeq I^z(\mu_h(z)) + \nabla_x I^z(\mu_h(z)) \cdot (\frac{\alpha}{k} - \mu_h(z)) + 
\langle  \text{Hess} I^z (\mu_h(z)) (\frac{\alpha}{k} - \mu_h(z)), \frac{\alpha}{k} - \mu_h(z) \rangle + O(k^{-3}). $$

As mentioned above,
$$I^z(\mu_h(z)) = 0, \; \nabla_x I^z |_{x = \mu_h(z)} = 0,
\;\;\text{Hess}\; I_z (\mu_h(z)) = \Hess \;u_\varphi(\mu_h(z)) = [\text{Hess} \;\varphi (\rho_z)]^{-1} = H_z^{-1}. $$
By Lemma \ref{pcalLem}, and by normalizing the weight, 
\begin{equation}k^{-m} \pcal_{h^k}(\alpha, z) = k^{-m/2} |\det H_z|^{-\half} e^{- k \langle  H_z^{-1} (\frac{\alpha}{k} - \mu_h(z)), \frac{\alpha}{k} - \mu_h(z) \rangle} (1+O(1/k)), \end{equation}
where $O(1/k)$ is uniform in $z, \alpha$. Distributing the $k$ in the exponent as $\sqrt{k}$ in each argument of the bilinear form completes the proof.

\end{document}